\documentclass[a4paper,11pt]{article}

\newtheorem{theorem}{Theorem}[section]

\newtheorem{corollary}[theorem]{Corollary}

\newtheorem{lemma}[theorem]{Lemma}

\newtheorem{proposition}[theorem]{Proposition}

\newenvironment{proof}[1][Proof]{\textbf{#1.} }{\ \rule{0.5em}{0.5em}}

\usepackage{amsfonts}
\usepackage{graphicx}
\usepackage{amsmath}
\usepackage[matrix,arrow]{xy}

\begin{document}

\title{Walks on Unitary Cayley Graphs and Applications}

\author{Elias Cancela, Daniel A. Jaume, Adri\'{a}n Pastine and Denis Videla\thanks{%
Department of Mathematics, Facultad de Ciencias
F\'{i}sico-Matem\'{a}ticas y Naturales, Universidad Nacional de San
Luis, Ej\'{e}rcito de los Andes 950, 5700 San Luis,
Argentina.\newline E-mail: djaume@unsl.edu.ar\newline
Fax: 54-2652-444014} \\
Universidad Nacional de San Luis}
\date{{\Large  Submitted to: The Electronic Journal of Combinatorics, on \today}
\bigskip \noindent \textbf{AMS subject classification: 05C25, 05C50.\quad
\qquad \qquad \qquad \\
\bigskip \noindent }}

\maketitle

\begin{abstract}
In this paper, we determine an explicit formula for the number of
walks in $X_n = \textsf{Cay}( \mathbb{Z}_n,\mathbb{U}_n )$, the
unitary Cayley Graphs of order $n$, between any pair of its
vertices. With this result, we give the number of representations of
a fixed residue class $\bmod{}n$ as the sum of $k$ units of
$\mathbb{Z}_n$.
\end{abstract}

\section{Introduction}

Let $\Gamma$ be a multiplicative group with identity $1$. For
$S\subset\Gamma$, $1\notin S$ and $S^{-1}=\{s^{-1}:s\in S\}=S$ the
\textit{Cayley Graph} $X=\textrm{Cay}\left(\Gamma,S\right)$ is the
undirected graph having vertex set $V\left(X\right)=\Gamma$ and edge
set $E\left(X\right)=\{{a,b}:ab^{-1}\in S\}$. By right
multiplication $\Gamma$ may be considered as a group of
automorphisms of $X$ acting transitively on $V(X)$. The Cayley graph
$X$ is regular of degree $|S|$. Its connected components are the
right cosets of the subgroup generated by $S$. So $X$ is connected,
if $S$ generates $\Gamma$. More information about Cayley graph can
be found in books on algebraic graph theory like those written by
Biggs \cite{B} and by Godsil and Royle \cite{GR}.

For a positive integer $n>1$ the \textit{unitary Cayley graph}
$X_{n}=\textrm{Cay}(\mathbb{Z}_{n},U_{n})$ is defined by the
additive group of the ring $\mathbb{Z}_{n}$ of integers modulo $n$
and the multiplicative group $\mathbb{U}_{n}$ of its units. If we
represent the elements of $\mathbb{Z}_{n}$ by the integers
$0,1,\ldots,n-1$ then it is well known that $U_{n}=\{r \in
\mathbb{Z}_{n}: \textrm{gcd}(r,n)=1\}$. So $X_{n}$ has vertex set
$V(X_{n})=\mathbb{Z}_{n}=\{0,1,\ldots,n-1\}$ and edge set
$E(X_{n})=\{\{a,b\}:a,b \in \mathbb{Z}_{n},
\textrm{gcd}(a-b,n)=1\}$.

Klotz and Sander in \cite{KS} show  that a graph $X_{n}$ is regular
of degree $|\mathbb{U}_{n}|=\varphi(n)$, where $\varphi(n)$ denotes
the Euler totient function. Let $p$ be a prime number, then
$X_{p}=K_{p}$ (the complete graph on $p$ vertices). Let $\alpha$ be
a positive integer, the $X_{p^{\alpha}}$ is a complete $p$-partite
graph which has the residue classes modulo $p$ in $\mathbb{Z}_{n}$
as maximal sets of independent vertices. Unitary Cayley graphs are
highly symmetric. They have some remarkable properties connecting
graph theory and number theory (for example it can be proved that
$\varphi(n)$ is even, for $n>2$, via a graph theory argument, using
unitary Cayley graphs).



\section{Walks in Complete Graphs}

Let $u$ and $v$ be (not necessarily distint) vertices of a graph
$X$. A $u-v$ \textit{walk} of $X$ is a finite, alternating sequence

\begin{displaymath}
u=u_0,e_1,u_1,e_2,\ldots,u_{k-1},e_k,u_k=v
\end{displaymath}
of vertices and edges, beginning with vertex $u$ and ending with
vertex $v$, such that $e_i=u_{i-1}u_i$ for $i=1,2,\ldots,k$. The
number $k$ is called the \textit{length} of the walk. A
\textit{trivial walk} contains no edges, that is, $k=0$. We note
that there may be repetition of vertices and edges in a walk. We
often will indicate only the edges of a walk.

A $u-v$ walk is \textit{closed} or \textit{open} depending on
whether $u=v$ or $u\neq v$.

We denote with $w(X,k,v_i,v_j)$ the total number of walk of length
$k$ between the vertices $v_i$ and $v_j$ of a given graph $X$.

In this section we are going to count the number of $k$-walks in
$K_n$, the complete graph of $n$ vertices.

 The \textit{adjacency matrix $A(X)$} of a graph $X$ is the
integer matrix with rows and columns indexed by the vertices of $X$,
such that the $uv$-entry of $A(X)$ is 1 if $u$ and $v$ are
neighbors, and 0 otherwise, so $A(X)$ is a symmetric 01-matrix.
Because a graph has no loops, the diagonal entries of $A(X)$ are
zero.

We will uses the next well known result:

\begin{theorem}\label{A}
If $A$ is the adjacency matrix of a graph $X$ with set of vertices
$V(X)=\{v_1,v_2,\ldots,v_n\}$, then the $(i,j)$ entry of $A^k$,
$k\geq1$, is the number of different $v_i-v_j$ walks of length $k$
in $X$.
\end{theorem}

In particular, we are interesting in the number of walks in complete
graphs. As $A(K_n)$ is the $n\times n$-matrix
\begin{displaymath}
A(K_n)= \textsf{circ}\left(0,1,\dots,1 \right)
\end{displaymath}
where $\textsf{circ}(0,1,\dots,1)$ is the circulant matrix whose
first row is the $n$-vector (0,1,\dots,1), see Davis \cite{D}. As
the circulant matrices are closed under product, we have that
\begin{displaymath}
A(K_n)^2= \textsf{circ} \left(n-1,n-2,\dots,n-2 \right)
\end{displaymath}
And in general
\begin{displaymath}
A(K_n)^k= \textsf{circ}\left(a_{n,k},b_{n,k},\dots, b_{n,k}
\right)
\end{displaymath}
where $a_{n,1}=0 $ $b_{n,1}=1$ and for $k\geq 2$
\begin{align}
a_{n,k} & = (n-1)b_{n,k-1} \nonumber \\
b_{n,k} & = a_{n,k}+(-1)^{k-1} \nonumber
\end{align}
Thus we have the following recursive relation
\begin{displaymath}
b_{n,k}=(n-1)b_{n,k-1}+(-1)^{k-1}
\end{displaymath}
which has the following closed form
\begin{displaymath}
b_{n,k}=\frac{1}{n}\left( \left(n-1 \right)^k-\left( -1 \right)^k
\right) \qquad \textrm{for} \qquad k \geq 0
\end{displaymath}

Therefore, we have the following result.
\begin{proposition}
The number of $v_i-v_j$ walks of length $k$ in a complete graphs
$K_n$ is
\begin{displaymath}
w(K_n,k,v_i,v_j):= \left\{
\begin{array} {ll}
\frac{1}{n}\left( \left(n-1 \right)^k-\left( -1 \right)^k \right) &
\textrm{if } v_i \neq v_j \\
\frac{n-1}{n}\left( \left(n-1 \right)^{k-1}-\left( -1 \right)^{k-1}
\right) & \text{if } v_i=v_j
\end{array}
\right.
\end{displaymath}

\end{proposition}

This result is known, see \cite{St}, but the proof given here is
different.



\section{Walks on $X_{\textsf{rad}(n)}$}

Given a positive integer $n$, we are going to count the number of
$k$-walks in $X_{\textsf{rad}(n)}$, the unitary Cayley graph of
$\textsf{rad}(n)$ vertices, where $\textsf{rad}(n)$ is the product
of all primes that divide $n$, i.e., $\textsf{rad}(n)=\prod_{p|n}p$.
So, we are going to resolve the problem for square-free integers.

Let $X_1=(V_1,E_1)$ and $X_2=(V_2,E_2)$ be graphs. The
\textit{Kronecker product} of $X_1$ and $X_2$ is the graph $X=(V,E)$
denoted by $X_1 \otimes X_2$ (also known as direct product,
categorical product, etc.) where $V=V_1 \times V_2$, the Cartesian
product of $V_1$ and $V_2$, with $\left( u_1, u_2 \right)$ and
$\left( v_1, v_2 \right)$ are adjacent in $X$ if and only if $u_1$
and $v_1$ are adjacent in $X_1$ and $u_2$ and $v_2$ are adjacent in
$X_2$. See \cite{HIK}.

A direct consequence of the definition of Kronecker product of
graphs is that a $k$-walk of $G \otimes H$ is the Kronecker product
of a $k$-walk of $G$ times a $k$-walk of $H$. Conversely, the
Kronecker product of a $k$-walk of $G$ times a $k$-walk of $H$ give
us a $k$-walk of $G \otimes H$. Summarizing

\begin{lemma}\label{LCK}
Given $n$  graphs $H_i$, let us consider the Kronecker product of
all them
\begin{displaymath}
G= \bigotimes_{i=1}^{n}H_i
\end{displaymath}
Any vertex $x$ of $G$ have the form $x=(x_1,x_2,\dots,x_n)$ where
$x_i \in V(H_i)$. Then the number of walk of length $k$ between any
two vertices $x$ and $y$ of $G$ is
\begin{displaymath}
w(G,k,x,y)=\prod_{i=1}^{n} w(H_i,k,x_i,y_i)
\end{displaymath}
\end{lemma}

For our next lemma we need the following two results of Ramaswamy
and Veena (2009) \cite{RV}
\begin{theorem}[Ramaswamy and Veena]
If $(m,n)=1$, then the Kronecker product of unitary Cayley graphs
$X_m$ and $X_n$ is isomorphic to $X_{mn}$.
\end{theorem}

\begin{corollary}[Ramaswamy and Veena]
If $n=p_{1}^{\alpha_1}p_{2}^{\alpha_2}\cdots p_{m}^{\alpha_m}$, then
the Kronecker product of unitary Cayley graphs $X_{p_{1}^{\alpha_1}}
\otimes X_{p_{2}^{\alpha_2}} \otimes \cdots \otimes
X_{p_{m}^{\alpha_m}}$ is isomorphic to $X_n$
\end{corollary}

Now we can given our next result
\begin{lemma}
Given posivite integers $n$ and $k$, for any $0 \leq i,j \leq n$ the
number of $i-j$-walks of length $k$ in $X_{\textsf{rad}(n)}$, i.e.,
$w(X_{\textsf{rad}(n)},k,i,j)$ is
\begin{equation}
\prod_{p\in C(n,i,j)}{\frac{p-1}{p}\left( \left(p-1
\right)^{k-1}-\left( -1 \right)^{k-1} \right)}\prod_{p \in
A(n,i,j)}\frac{1}{p}\left( \left(p-1 \right)^k-\left( -1 \right)^k
\right)
\end{equation}
where
\begin{align}
C(n,i,j)= \{ p \text{ prime }: p|n, {} i=j\bmod p \} \\
A(n,i,j)=\{ p \text{ prime }: p|n, {} i \neq j\bmod p \}
\end{align}
\end{lemma}


\begin{proof}
Direct from the lemma \ref{LCK} and the previous corollary.
\end{proof}



\section{Relation between $X_n$ and $X_{\textsf{rad}(n)}$}

We use the blow-up notation. Given a graph $G$ and a positive
integer $n$, The blow-up of $G$ of order $r$, denote by $B(G,r)$, is
the graph with set vertex $V(B(G,r))=V(G)\times[r]$, where as usual
$[r]:=\{1,2,\dots,r\}$, and set of edges
\begin{displaymath}
E(B(G,r))=\{ \{(u,a),(v,b)/ \{u,v\}\in E(G), a,b \in [r]\}
\end{displaymath}
i.e., the vertex $(u,a)$ is neighbor of the vertex $(v,b)$ in
$B(G,r)$ if and only if the vertices $v$ and $u$ of $G$ are
neighbors.

\begin{proposition}\label{PWB}
Given positive integers $r$ and $k$, and a graph $G$, the number of
$k$-walks in $B(G,r)$ between its vertices $(u,a)$ and $(v,b)$ is:
\begin{displaymath}
w(B(G,r),k,(u,a),(v,b))=r^{k-1}w(G,k,u,v)
\end{displaymath}
\end{proposition}

\begin{proof}
Evident from the next observation: given a $k$-walk in $G$ between
its vertices $u$ and $v$:
\begin{displaymath}
u=v_0,v_1,\dots,v_k=v
\end{displaymath}
We have that
\begin{displaymath}
(u,a)=(v_0,a_0),(v_1,a_1),\dots,(v_k,a_k)=(v,b)
\end{displaymath}
with $a_i \in [r]$, is a $k$-walk in $B(G,r)$ for any choice of the
integers $a_i$. Thus if we fix the ends, i.e. if we fix $(u,a)$ and
$(v,b)$, we have that the number of $k$-walks between them in
$B(G,r)$ is $r^{k-1}w(G,k,u,v)$.
\end{proof}


Now we are going to proved that all the structural information of
$X_n$ is in $X_{\textsf{rad}(n)}$, actually we will prove that
$X_n=B(X_{\textsf{rad}(n)},{} n/\textsf{rad}(n))$.

For the next proposition we recall that for each vertex $x$ of a
graph $G$, the neighborhood of $x$ in $G$ is $N_G(x)=\{y \in V(G):
\{x,y\} \in E(G) \}$.

\begin{theorem}
Given $x,y \in \mathbb{Z}_n$ if $x \equiv y \bmod \textsf{rad}(n)$
then $x$ and $y$ are not neighbors in $X_n$ and they have the same
neighborhood in $X_n$: $N(x)=N(y)$
\end{theorem}
\begin{proof}
The first statement is a clear consequence of the following
elementary facts:
    \begin{enumerate}
        \item In an unitary Cayley graphs two vertices $x$ and $y$
            are neighbors if and only if $x-y \in \mathbb{U}_n$.
        \item For any two integers $a$ and $b$, $(a,b)=1$ if
            and only if $(a,\textsf{rad}(b))=1$.
        \item For any two integers $a$ and $b$, if $a \equiv 0 \bmod
        b$, then $a \equiv 0 \bmod \textsf{rad}(b)$, i.e., $(a,\textsf{rad}(b)) \neq 1$.
    \end{enumerate}

The set of neighbors of $x$ is
\begin{equation}
N(x):= \left\{ z \in \mathbb{Z}_n : z-x \in \mathbb{U}_n\right\}
\nonumber
\end{equation}
As $x \equiv y \bmod rad(n)$, then for all $z \in \mathbb{Z}_n$ such
that $z-x \in \mathbb{U}_n$ we have that $z-x \equiv z-y \bmod
rad(n)$, then $(z-y,rad(n))=1$, i.e., $z-y \in \mathbb{U}_n$, so
$N(x)=N(y)$.
\end{proof}

\begin{corollary}\label{CB}
For any positive integer $n$ we have that the unitary Cayley graph
$X_n$ is the blow-up of order $n/\textsf{rad}(n)$ of
$X_{\textsf{rad}(n)}$:
\begin{equation}
X_n=B \left( X_{\textsf{rad}(n)},\frac{n}{\textsf{rad}(n)}\right)
\end{equation}
\end{corollary}

\begin{proof}
Obvious from the previous theorem and definition of blow-up.
\end{proof}

Now we can set our main result:

\begin{theorem}\label{Main}
Given two positive integers $n$ and $k$, the number of $i-j$-walks
of length $k$ in $X_{n}$ is
    \begin{equation}
        w(X_{n},k,i,j)=\left(
        \frac{n}{\textsf{rad}(n)}\right)^{k-1}w(X_{\textsf{rad}(n)},k,i \bmod {\textsf{rad}(n)},j \bmod {\textsf{rad}(n)})
    \end{equation}
where $i \bmod {\textsf{rad}(n)}$ is the residue of $i$ module
$\textsf{rad}(n)$.
\end{theorem}

\begin{proof}
Just use corollary \ref{CB} and proposition \ref{PWB}
\end{proof}


\section{Applications}

We give two direct applications of theorem \ref{Main}. The first one
to additive number theory and the second one to linear algebra.

First application: the multiplicative groups of units in the ring
$\mathbb{Z}_n$ of residue classes mod $n$ consists of the residues
$r$ mod $n$ with $(r,n)=1$. We use unitary Cayley graphs to
determine the number of representations of a fixed residue class mod
$n$ as the sum of $k$ units in $\mathbb{Z}_n$

In this section we are going to show that there exists an obvious
bijection between walks in $X_n$ and ordered sums of units in
$\mathbb{Z}_n$.

Let $k$ and $n$ be positive integers, and let $r$ be a residue class
of $\mathbb{Z}_n$. We call the set of all ordered $k-$sums of units
of $\mathbb{Z}_n$ that sum $r$ module $n$ as $S(n,k,r)$, so
\begin{equation}
S(n,k,r):= \left\{(u_1,u_2,\ldots,u_k) \in \mathbb{U}_{n}^{k}:\quad
u_1+u_2+\cdots+u_k \equiv r \bmod n \right\} \label{problema}
\end{equation}
The cardinality of the above set will be denote with $s(n,k,r):=|
S(n,k,r)|$ .

The next result says that each ordered $k$-sum of units of
$\mathbb{Z}_n$ that sums $r$ give us a $k$-walk between $0$ and $r$
in $X_n$, and conversely.

\begin{theorem}
For all positive integers $n$ and $k$, and for any residue class $r$
of $\mathbb{Z}_n$, we have that $s(n,k,r)=w(X_n,k,0,r)$.
\end{theorem}
\begin{proof}
Given $(u_1,u_2,\ldots,u_k) \in S(n,k,r)$, we can obtain a $k$-walk
from $0$ to $r$ in $X_n$ as follow:
    \begin{equation}
    0,u_1,0+u_1,u_2,0+u_1+u_2, \dots,u_k,0+u_1+u_2+\cdots+u_k=r
    \nonumber
    \end{equation}
Conversely, if we have a $k$-walk from $0$ to $r$ in $X_n$:
    \begin{equation}
    0,u_1,x_1,u_2,x_2, \dots,u_k,x_k=r
    \nonumber
    \end{equation}
by definition of $X_n$ for each $i \in {1,\ldots,k}$
    \begin{equation}
    x_i =0+u_1+u_2+\cdots+u_i
    \nonumber
    \end{equation}
Thus $u_1+u_2+\cdots+u_k = x_k=r$, and the edges of the $k$-walk
over $X_n$ gives an ordered $k$-sum of units of $\mathbb{Z}_n$ that
sums $r$.
\end{proof}

These result is a generalization of part of a recent work by Sander
(2009) \cite{S}, who proved it in the case $k=2$.

In particular we have that for any two positive integers $n$ and
$k$, the number of ordered homogeneous $k$-sums in $\mathbb{Z}_{n}$,
is
\begin{equation}
s(n,k,0)=\left( \frac{n}{rad(n)}\right)^{k-1}
        \prod_{p|n}{\frac{p-1}{p}\left( \left(p-1
        \right)^{k-1}-\left( -1 \right)^{k-1} \right)}
\end{equation}

This last result, could be prove without using unitary Cayley
graphs, via a recursive argument.

Second application: Circulant matrices arise, for example, in
applications involving the discrete Fourier transformation and the
study of cyclic codes for error correction \cite{D}.

In Rimas\cite{R1} (2005), \cite{R2} (2005), \cite{R3} (2006),
\cite{R4} (2006), and K\"{o}ken and Bozkurt \cite{KB} (2011) these
authors have studied positive integer powers for circulant matrices
of type $\textsf{circ}(0,a,0,\dots,b)$.

Now, using that power of circulat matrices are circulant, see
\cite{D}, and theorem \ref{Main} we can easily derive a general
expression for the arbitrary positive powers of circulant matrices
that are adjancency matrices of unitary Cayley Graphs, by a very
different technique for that used in the works by Rimas, or
K\"{o}ken and Bozkurt.

From theorems \ref{Main} and \ref{A}:

\begin{proposition}
Given $n,k \in \mathbb{Z}$, we have that
\begin{displaymath}
A(X_{n})^k=\textsf{circ}(w(X_{n},k,0,0),w(X_{n},k,0,1),\dots,w(X_{n},k,0,n))
\end{displaymath}
\end{proposition}





\begin{thebibliography} {9}



\bibitem{B} {\sc{N. Biggs}}, \textit{Algebraic graph theory}, Second Edition,
Chematical library. Cambridge University Press, 1993.

\bibitem{D} {\sc{J. P. Davis}}, \textit{Circulant Matrices}. Chelsea
Publishing, New York. 1994.

\bibitem{GR} {\sc{C. Godsil and R. Royle}}, {\em Algebraic graph theory}, Graduate Text in Mathematics. Springer, 2001.

\bibitem{HIK} \sc{R. Hammack, W. Imrich and S. Klav\v{z}ar}, {\em Handbook of Product
Graphs}. 2nd edition. CRC Press. 2011.

\bibitem{KS}  {\sc{W. Klotz and T. Sander}}, {\em Some properties of unitary Cayley graphs}, The Electronic Journal of Combiantorics
14 (2007), R45, pp. 1-12.

\bibitem{KB} F. K\"{o}ken and D. Bozkurt, \textit{Positive integer
powers for one type of odd order circulant matrices}. Applied
Mathematics and Computation 217 (2011) 4377-4381.

\bibitem{RV} H.N. Ramaswamy and C.R. Veena, \textit{On the Energy of Unitary
Cayley Graphs. }The Electronic Journal of Combinatorics 16 (2009),
\#N24

\bibitem{R1} J. Rimas, \textit{On computing of arbitrary positive
integer powers for one type of odd order symmetric circulant
matrices-I} Applied Mathematics and Computation 165 (2005) 137-141.

\bibitem{R2} J. Rimas, \textit{On computing of arbitrary positive
integer powers for one type of odd order symmetric circulant
matrices-II} Applied Mathematics and Computation 169 (2005)
1016-1027.

\bibitem{R3} J. Rimas, \textit{On computing of arbitrary positive
integer powers for one type of even order symmetric circulant
matrices-I} Applied Mathematics and Computation 172 (2006) 86-90.

\bibitem{R4} J. Rimas, \textit{On computing of arbitrary positive
integer powers for one type of even order symmetric circulant
matrices-II} Applied Mathematics and Computation 174 (2006) 511-552.

\bibitem{S} J. W. Sander, \textit{On the addition of units and nonunits mod
m}, Journal of Number theory 129 (2009) 2260-2266.

\bibitem{St} {\sc{R. P. Stanley}} {\em Topic in Algebraic
Combinatorics}. Preliminary version 24th April 2010.

\end{thebibliography}
\end{document}